\documentclass[reqno,12pt]{amsart}

\usepackage[utf8]{inputenc}
\usepackage{microtype}
\usepackage{amsfonts}
\usepackage{amsmath}
\usepackage{graphicx}
\usepackage{float}
\usepackage[dvipsnames]{xcolor}
\usepackage{hyperref}
\usepackage{tikz-cd}

\usepackage{amssymb}
\usepackage{enumitem}
\usepackage{mathrsfs}

\usepackage{tikz}
\usetikzlibrary{topaths}
\usetikzlibrary{calc}

\usepackage{a4wide}

\usepackage{empheq}

\newtheorem{theorem}{Theorem}[section]
\newtheorem*{theorem*}{Theorem}
\newtheorem{lemma}[theorem]{Lemma}
\newtheorem{proposition}[theorem]{Proposition}
\newtheorem*{proposition*}{Proposition}
\newtheorem{corollary}[theorem]{Corollary}
\newtheorem*{corollary*}{Corollary}

\newtheorem{cit}[theorem]{Citation}
\newtheorem*{conjecture*}{Conjecture}

\newtheorem*{question*}{Question}

\newtheorem*{main:exist_simple}{Theorem~\ref{thrm:exist_simple}}

\theoremstyle{definition}
\newtheorem{definition}[theorem]{Definition}

\newcommand{\N}{\mathbb{N}}
\newcommand{\Z}{\mathbb{Z}}

\newcommand{\tree}{\mathcal{T}}

\DeclareMathOperator{\Aut}{Aut}

\DeclareMathOperator{\Homeo}{Homeo}

\DeclareMathOperator{\Area}{Area}

\DeclareMathOperator{\F}{F}

\numberwithin{equation}{section}

\begin{document}

\title{Finitely presented simple groups with at least exponential Dehn function}
\date{\today}
 %trees


\author[M.~C.~B.~Zaremsky]{Matthew C.~B.~Zaremsky}
\address{Department of Mathematics and Statistics, University at Albany (SUNY), Albany, NY}
\email{mzaremsky@albany.edu}

\begin{abstract}
We construct examples of finitely presented simple groups whose Dehn functions are at least exponential. To the best of our knowledge, these are the first such examples known. Our examples arise from R\"over--Nekrashevych groups, using carefully calibrated self-similar representations of Baumslag--Solitar groups.
\end{abstract}

\maketitle
\thispagestyle{empty}

\section*{Introduction}

Finitely presented simple groups have enjoyed a recent surge of interest from a geometric and topological standpoint. Caprace and R\'emy proved in \cite{caprace09,caprace10} that there exist infinitely many quasi-isometry classes of finitely presented simple groups. In \cite{skipper19}, Skipper, Witzel, and the author found examples of finitely presented simple groups with arbitrary finiteness length. Hyde and Lodha recently found examples of finitely presented simple groups that are left-orderable \cite{hyde23}. This followed a great deal of work constructing and analyzing finitely generated left-orderable simple groups \cite{hyde19,mattebon20,hyde21,fournierfacio23}. Another recent geometric result about finitely generated simple groups, proved by Belk and the author \cite{belk22}, and independently by Darbinyan and Steenbock \cite{darbinyan22}, is that every finitely generated group isometrically embeds as a subgroup of a finitely generated simple group.

In this paper we investigate another topic of interest in geometric group theory, namely Dehn functions, applied to finitely presented simple groups. The Dehn function of a finitely presented group measures how many instances of the defining relations are needed to realize every relation of a given length. It is well known that a finitely presented group has solvable word problem if and only if its Dehn function is recursively defined, and in general the Dehn function can be viewed as a geometric measurement of how difficult it is to solve the word problem (or at least as a sort of upper bound).

Finitely presented simple groups have solvable word problem \cite{kuznetsov58}, hence recursively defined Dehn function. However, among the existing examples of finitely presented simple groups where something is known about their Dehn function, the function is not only recursive but polynomial, i.e., very small. More precisely, to the best of the author's knowledge, the only finitely presented simple groups where something is known about their Dehn functions are the Burger--Mozes groups \cite{burger00}, which have quadratic Dehn function thanks to acting geometrically on a product of trees, and Thompson's groups $T$ and $V$, which are known to have polynomial Dehn function (the best bounds we are aware of are $\delta_T\preceq n^5$ \cite{wang15} and $\delta_V\preceq n^{11}$ \cite{guba00}, and one would conjecture that they are quadratic, like for Thompson's group $F$ \cite{guba06}). Presumably, close relatives like the Higman--Thompson groups $T_d$ and $V_d$, which are virtually simple, also have polynomial Dehn functions, using similar arguments. Thus, all known examples have ``very small'' Dehn function. It is worth mentioning the Brin--Thompson groups $nV$ \cite{brin04}, which are finitely presented and simple, and whose Dehn functions are unknown; it turns out that if $nV$ ($n\ge 2$) has polynomial Dehn function (or even embeds in a finitely presented group with polynomial Dehn function), then $\sf{NP}=\sf{coNP}$ \cite{birget20}.

\medskip

In this paper, we construct examples of finitely presented simple groups whose Dehn functions are strictly larger than polynomial, namely they are at least exponential. The main tool is the family of R\"over--Nekrashevych groups $V_d(G)$ of self-similar groups $G$, introduced in \cite{roever99,nekrashevych04}. Our specific examples are denoted
\[
[V_{n+2}(BS(1,n)),V_{n+2}(BS(1,n))] \text{,}
\]
the notation for which we now unpack. The group $BS(1,n)$ is the usual Baumslag--Solitar group $BS(1,n)=\langle a,b\mid aba^{-1}=b^n\rangle$ for $n\ge 2$, which has exponential Dehn function \cite{gersten92}. We find a certain self-similar representation of $BS(1,n)$ acting on the infinite rooted $(n+2)$-ary tree $\tree_{n+2}$, inspired by Bartholdi and Sunic's self-similar representation of $BS(1,n)$ acting on $\tree_{n+1}$ from \cite{bartholdi06}. This representation is calibrated to have a variety of properties, inspired by \cite{skipper19}, which ensure that, among other things, the R\"over--Nekrashevych group $V_{n+2}(BS(1,n))$ has $BS(1,n)$ as a quasi-retract, and its commutator subgroup has finite index and is simple. Putting all of this together yields our main result:

\begin{main:exist_simple}
There exist finitely presented simple groups with at least exponential Dehn function.
\end{main:exist_simple}

One might also like to find an upper bound, and presumably say that these examples have precisely exponential Dehn function, but for a couple reasons we do not approach this here. On the one hand, this would involve completely different techniques than those used here, and on the other hand, this deserves to be part of a broader program to find upper bounds on Dehn functions of arbitrary R\"over--Nekrashevych groups. In particular, we suspect that the Dehn function of any $V_d(G)$ should be bounded above by some combination of the Dehn functions of $V_d(\{1\})$ and $G$, and this could perhaps be approached by looking at the action of $V_d(G)$ on a simply connected, cocompact truncation of the Stein--Farley complex of $V_d(G)$ (see, e.g., \cite{skipper21}). In any case, this is all beyond the scope of the present paper.

It would be very interesting to find examples of finitely presented self-similar groups $G$ with even larger Dehn functions, which could then perhaps lead to finitely presented simple groups $[V_d(G),V_d(G)]$ with even larger Dehn functions, using the techniques here. There is a restriction though, that self-similar implies residually finite, and examples of residually finite groups with large Dehn function were only recently found by Kharlampovich, Myasnikov, and Sapir in \cite{kharlampovich17} using some very complicated constructions. We do not know whether their groups admit faithful self-similar representations. Another, easier, example of groups with very large Dehn function comes from the ``hydra groups'' of Dison and Riley \cite{dison13}, but Pueschel proved that these are not residually finite \cite{pueschel16}. One could also try to bypass self-similarity by looking for examples among twisted Brin--Thompson groups \cite{belk22}, which have similar properties to R\"over--Nekrashevych groups but have the advantage that the input group $G$ can be any group, not necessarily self-similar. The downside is that, unlike for R\"over--Nekrashevych groups, finite presentability of the twisted Brin--Thompson group does not follow for free from finite presentability of $G$, and in fact is rather difficult to achieve.

As a remark, one reason to expect that there exist finitely presented simple groups with large (perhaps even arbitrarily large recursive) Dehn function is the Boone--Higman conjecture, which predicts that every finitely generated group with solvable word problem embeds in a finitely presented simple group \cite{boone74}. If this holds, then embedding a group with arbitrarily difficult, solvable word problem into a finitely presented simple group would provide an arbitrarily large, recursive lower bound on the Dehn function of the simple group. The solution to the word problem in \cite{kuznetsov58} for finitely presented simple groups does not give any particular uniform upper complexity bound, so a priori there is not any reason to doubt that arbitrarily large, recursive Dehn functions are possible. See \cite{bbmz_survey} for more background on the Boone--Higman conjecture.

\medskip

This paper is organized as follows. In Section~\ref{sec:dehn} we recall some background on Dehn functions and quasi-retracts. In Section~\ref{sec:nekr} we discuss self-similar groups and R\"over--Nekrashevych groups, along with the various properties of self-similar actions that will lead to our results. Finally, in Section~\ref{sec:examples} we construct our examples.

\subsection*{Acknowledgments} Thanks are due to Emmanuel Rauzy and Giles Gardam for helpful comments and pointing out references. I am also grateful to the referee for several excellent suggestions. The author is supported by grant \#635763 from the Simons Foundation.

%-------------------------------------------------
\section{Dehn functions and quasi-retracts}\label{sec:dehn}

In this section we recall some background on Dehn functions and quasi-retracts.

\subsection{Dehn functions}\label{ssec:dehn}

We will not need to use too many details about Dehn functions, so we just give a quick definition and overview following \cite[Section~I.8A.4]{bridson99}. Let $G=\langle S\mid R\rangle$ be a finite presentation, so $S$ is a finite set, $R$ is a finite subset of the free group $F(S)$, and $G$ is the quotient of $F(S)$ by the normal closure of $R$. Write $\pi\colon F(S)\to G$ for this quotient map. An element of $\ker(\pi)$ can be written as a product of elements of $S^\pm$, or as a product of elements of the set $(R^\pm)^{F(S)}$ of conjugates of $R^\pm$ (here we write $X^\pm$ to mean the union of $X$ with the set of inverses of elements of $X$). Roughly speaking, the Dehn function of $G$ measures how different the lengths of these expressions can be.

To be more precise, for $w\in\ker(\pi)$ define the \emph{area} $\Area(w)$ of $w$ to be the word length of $w$ in the (likely infinite) generating set $(R^\pm)^{F(S)}$ of $\ker(\pi)$. Also define the \emph{length} $\ell(w)$ of $w$ to be its usual word length in the generating set $S$ of $F(S)$. Now the \emph{Dehn function} of $G$ is the function $\delta_G\colon \N\to\N$ defined via
\[
\delta_G(n) \coloneqq \max\{\Area(w)\mid w\in\ker(\pi)\text{, }\ell(w)\le n\}\text{.}
\]
Note that for any $n$, only finitely many $w$ have length at most $n$, so $\delta_G$ is well defined. (As a remark, up until now we have not actually needed $R$ to be finite, but this is a necessary assumption for various upcoming results to be true.)

The function $\delta_G$ as defined depends on the choice of finite presentation for $G$, so we tend to consider Dehn functions up to a certain equivalence relation. Given two functions $f,g\colon \N\to\N$, write $f\preceq g$ if there is a constant $K>0$ such that for all $n\in\N$ we have $f(n)\le Kg(Kn)+Kn$. If $f\preceq g$ and $g\preceq f$, write $f\simeq g$. This is an equivalence relation, and it turns out that the Dehn functions arising from two finite presentations of the same group are equivalent. More generally, the Dehn functions of any two quasi-isometric finitely presented groups are equivalent.

Note that if $f\simeq g$ then certain features are common to both $f$ and $g$. For instance, if one is linear then so is the other, and more generally if one is a polynomial of degree $m\ge 1$ then so is the other. If one of $f$ or $g$ is an exponential function, then so is the other, perhaps with a different base (for example $f(n)=2^n$ and $g(n)=3^n$ are equivalent since $3^n=2^{\log_2(3)n}$). Thus it makes sense to say that a finitely presented group, ``has a polynomial Dehn function,'' or, ``has an exponential Dehn function.''

\subsection{Quasi-retracts}\label{ssec:qr}

The proof that Dehn functions of quasi-isometric finitely presented groups are equivalent, which for example is in \cite{alonso90}, actually shows that if $H$ is a so called quasi-retract of $G$, then $\delta_H \preceq \delta_G$. This is implicit in \cite{alonso90}, and is stated explicitly for example in \cite[Theorem~3]{alonso96}. Let us recall the details of quasi-retracts now.

A function $f\colon X\to Y$ between metric spaces, with metrics $d_X$ and $d_Y$ respectively, is called \emph{$(C,D)$-Lipschitz} for $C\ge 1$ and $D\ge 0$ if for all $x,x'\in X$ we have
\[
d_Y(f(x),f(x'))\le C d_X(x,x') + D \text{.}
\]

\begin{definition}[Quasi-retract(ion)]
Let $X$ and $Y$ be metric spaces, with metrics $d_X$ and $d_Y$ respectively. If there exist $(C,D)$-Lipschitz functions $r\colon X\to Y$ and $\iota\colon Y\to X$ such that $d_Y(r\circ\iota(y),y)\le D$ for all $y\in Y$, then we call $r$ a \emph{quasi-retraction}, and call $Y$ a \emph{quasi-retract} of $X$.
\end{definition}

\begin{cit}\cite{alonso90,alonso96}\label{cit:qr_dehn}
Let $G$ and $H$ be finitely presented groups, viewed as metric spaces via word metrics coming from finite generating sets. Suppose $H$ is a quasi-retract of $G$. Then $\delta_H \preceq \delta_G$.
\end{cit}

As a remark, if a pair of functions satisfy the quasi-retraction condition when composed in either order, then they are quasi-isometries. Thus, we get that Dehn functions are a quasi-isometry invariant of finitely presented groups.

%-------------------------------------------------
\section{Self-similar groups and R\"over--Nekrashevych groups}\label{sec:nekr}

In this section we discuss the source of our examples of finitely presented simple groups. Let $\tree_d$ be the infinite rooted $d$-ary tree. We identify the vertex set of $\tree_d$ with the set $\{1,\dots,d\}^*$ of finite words in the alphabet $\{1,\dots,d\}$; the root is the empty word $\varnothing$. Two vertices are adjacent if they are of the form $w$ and $wi$ for some $i\in\{1,\dots,d\}$. An \emph{automorphism} of $\tree_d$ is a bijection from $\{1,\dots,d\}^*$ to itself that preserves adjacency.

Now consider the group $\Aut(\tree_d)$ of automorphisms of $\tree_d$. Since the root is the only vertex of degree $d$, every automorphism preserves the measurement ``distance to root''. In particular the set $\{1,\dots,d\}$ of the children of the root is stabilized by every automorphism. This gives us a surjective homomorphism $\rho_d\colon \Aut(\tree_d)\to S_d$, which clearly splits. The kernel of $\rho_d$ is isomorphic to $\Aut(\tree_d)^d$, and so we get a wreath product decomposition
\[
\Aut(\tree_d) \cong S_d \wr \Aut(\tree_d) \text{.}
\]
Note that, for future convenience, we write wreath products with the acting group on the left, and view $\Aut(\tree_d)$ acting on $\tree_d$ as a left action.

For $g\in \Aut(\tree_d)$, the \emph{wreath recursion} of $g$ is the identification $g\leftrightarrow\rho_d(g)(g_1,\dots,g_d)$ induced by this isomorphism. The automorphisms $g_i$ are called the \emph{level-1 states} of $g$. The \emph{states} of $g$ are the elements of the smallest set containing $g$ that is closed under taking level-1 states.

\begin{definition}[Self-similar]
Call a subgroup $G\le\Aut(\tree_d)$ \emph{self-similar} if for all $g\in G$, every state of $g$ is in $G$. (Equivalently, for all $g\in G$, every level-1 state of $g$ is in $G$.)
\end{definition}

For a wealth of background on self-similar groups, see \cite{nekrashevych05}. Note that sometimes in the literature ``self-similar'' requires $\rho_d(G)$ to act transitively on $\{1,\dots,d\}$, but we do not require this. We will also refer to a \emph{self-similar action} of a group, which is a homomorphism from the group to $\Aut(\tree_d)$ whose image is self-similar. When we call a group self-similar, we are really implicitly referring to a fixed faithful self-similar action of the group. Given a group $G$ together with a declared wreath recursion for each element of some generating set, we get a well defined self-similar action of $G$, assuming that the defining relations of $G$ are satisfied by the wreath recursions.

\begin{definition}[Rational]
An element of a self-similar group is \emph{rational} (or \emph{finite-state}) if it has finitely many states. Call the group itself \emph{rational} if every element is rational.
\end{definition}

If every generator of the group is rational, then the same is true of every element, so it suffices to check rationality on the elements of some choice of generating set.

\medskip

Now we shift focus to the boundary of $\tree_d$, which is the $d$-ary Cantor set $C_d=\{1,\dots,d\}^\N$. For each $w\in\{1,\dots,d\}^*$, the \emph{cone} on $w$ is the basic open set $C_d(w)\coloneqq \{w\kappa\mid \kappa\in C_d\}$ in $C_d$. Any cone is canonically homeomorphic to $C_d$, via the \emph{canonical homeomorphism}
\[
h_w \colon C_d \to C_d(w)
\]
sending $\kappa$ to $w\kappa$.

\begin{definition}[R\"over--Nekrashevych group]
Let $G\le \Aut(\tree_d)$ be self-similar. The \emph{R\"over--Nekrashevych group} $V_d(G)$ is the subgroup of $\Homeo(C_d)$ consisting of all homeomorphisms constructed as follows:
\begin{enumerate}
    \item Partition $C_d$ into finitely many cones $C_d(w_1^+),\dots,C_d(w_n^+)$.
    \item Partition $C_d$ into the same number of cones in some possibly different way $C_d(w_1^-),\dots,C_d(w_n^-)$.
    \item Map $C_d$ to itself by sending each $C_d(w_i^+)$ to some $C_d(w_j^-)$ via the map $h_{w_j^-}\circ g_i\circ h_{w_i^+}^{-1}$ for some $g_i\in G$.
\end{enumerate}
\end{definition}

In words, an element of $V_d(G)$ acts on a cone $C_d(w_i^+)$ in the domain partition by removing the old prefix $w_i^+$, then acting via some element of $G$, and then adding a new prefix $w_j^-$. The self-similarity condition ensures that $V_d(G)$ is closed under compositions, and so really is a group. R\"over--Nekrashevych groups were introduced first by R\"over in \cite{roever99} for the special case when $G$ is the Grigorchuk group from \cite{grigorchuk80,grigorchuk84}, and in generality by Nekrashevych in \cite{nekrashevych04}. If $G$ is finitely generated, then so is $V_d(G)$, and if $G$ is finitely presented, then so is $V_d(G)$ (more generally if $G$ is of type $\F_n$ then so is $V_d(G)$ \cite[Theorem~4.15]{skipper19}). When $G=\{1\}$, the R\"over--Nekrashevych group $V_d(\{1\})$ is the classical Higman--Thompson group $V_d$.

\begin{definition}[Weakly diagonal]
Call a self-similar group $G\le\Aut(\tree_d)$ \emph{weakly diagonal} if there exists a generating set $S$ for $G$ such that for all $s\in S$ with wreath recursion $s\leftrightarrow\rho_d(s)(s_1,\dots,s_d)$, each $s_i$ satisfies that $s_is^{-1}[G,G]$ has finite order in the abelianization $G/[G,G]$. Call a self-similar action \emph{weakly diagonal} if its image is.
\end{definition}

In \cite{skipper19}, the key property of a self-similar group $G$ that ensured virtual simplicity of $V_d(G)$ was being ``coarsely diagonal''---this means that for all $g\in G$, for any level-1 state $g_i$, the element $g^{-1}g_i$ has finite order (which is equivalent to $g_i g^{-1}$ having finite order). Our notion of weakly diagonal here is a much weaker condition, since it only requires a condition on generators, and only requires finite order in the abelianization. The examples we will construct later will not be coarsely diagonal, so we really need this new notion of weakly diagonal. As we now see, it is still sufficient to ensure virtual simplicity.

\begin{lemma}\label{lem:v_simple}
If $G\le\Aut(\tree_d)$ is self-similar and weakly diagonal, then $V_d(G)$ is virtually simple. More precisely, the commutator subgroup $[V_d(G),V_d(G)]$ is simple and has finite index in $V_d(G)$.
\end{lemma}

\begin{proof}
Nekrashevych proved that the commutator subgroup $[V_d(G),V_d(G)]$ is always simple \cite[Theorem~4.7]{nekrashevych04}, so it suffices to prove that $V_d(G)$ has finite abelianization. We will roughly follow the strategy from the proof of \cite[Theorem~3.3]{skipper19}, where it was assumed that $G$ is coarsely diagonal (here we only assume it is weakly diagonal). The group $V_d(G)$ is generated by the Higman--Thompson group $V_d=V_d(\{1\})$ together with a certain copy of $G$, which we now explain. Let $\iota_1\colon V_d(G)\to V_d(G)$ send $h$ to the homeomorphism defined by $\iota_1(h)(1\kappa)\coloneqq 1h(\kappa)$ and $\iota_1(h)(i\kappa)=i\kappa$ for all $2\le i\le d$ and all $\kappa\in C_d$. In words, $\iota_1(h)$ acts like $h$ on the cone $C_d(1)$ and acts trivially everywhere else. More generally, for any $w\in \{1,\dots,d\}^*$, let $\iota_w\colon V_d(G)\to V_d(G)$ send $h$ to the homeomorphism that acts like $h$ on $C_d(w)$ and trivially everywhere else. It is easy to see that $\iota_w(h)$ is conjugate in $V_d(G)$ to $\iota_{w'}(h)$, via conjugation by an element of $V_d$, for any $h\in V_d(G)$ and any non-empty $w$ and $w'$. Now recall that $V_d(G)$ is generated by $V_d$ and $\iota_1(G)$; this is \cite[Lemma~5.11]{nekrashevych18}, and see \cite[Lemma~3.4]{skipper19} for the result using our current notation. Intuitively, the reason is that, by construction $V_d(G)$ is generated by $V_d$ together with all the $\iota_w(G)$ for $w\in \{1,\dots,d\}^*$, then thanks to wreath recursion we may assume $w$ is non-empty, and finally thanks to conjugation by $V_d$ we only need $\iota_1(G)$.

Now it suffices to prove that every generator of $V_d(G)$ from the generating set $V_d\cup\iota_1(G)$ has finite order in the abelianization. The group $V_d$ is virtually simple, so elements of $V_d$ have finite order in the abelianization. Since $G$ is weakly diagonal, we can choose a generating set $S$ for $G$ satisfying the property from the definition of weakly diagonal. Now $\iota_1(G)$ is generated by $\iota_1(S)$, so we want to show that every element of $\iota_1(S)$ has finite order in the abelianization of $V_d(G)$. For $s\in S$, let $s\leftrightarrow \rho_d(s)(s_1,\dots,s_d)$ be the wreath recursion of $s$. Using the $\iota_w$ maps, this is the same as $s=\rho_d(s) \iota_1(s_1)\cdots\iota_d(s_d)$. Since $\iota_1$ is a homomorphism, and clearly $\iota_w\circ \iota_{w'}=\iota_{ww'}$ for any $w$ and $w'$, we get $\iota_1(s)=\sigma \iota_{11}(s_1)\cdots\iota_{1d}(s_d)$ for some permutation $\sigma$. Since each $\iota_{1i}(s_i)$ is conjugate to $\iota_1(s_i)$, we conclude that $\iota_1(s_1)\cdots\iota_1(s_d)\iota_1(s)^{-1}$ has finite order in the abelianization, namely, its order divides the order of $\sigma$. Next note that by weak diagonality each $s_i s^{-1}$ has finite order in the abelianization, and since $\iota_1$ is a homomorphism the same is true of each $\iota_1(s_i)\iota_1(s)^{-1}$. Now multiplying $\iota_1(s_1)\cdots\iota_1(s_d)\iota_1(s)^{-1}$ by the inverse of each $\iota_1(s_i)\iota_1(s)^{-1}$, we conclude that $\iota_1(s)^{d-1}$ has finite order in the abelianization. Since $d\ge 2$, this implies that $\iota_1(s)$ has finite order in the abelianization, as desired.
\end{proof}

\begin{definition}[Persistent]
Call a self-similar group $G\le \Aut(\tree_d)$ \emph{persistent} if for all $g\in G$, in the wreath recursion $g\leftrightarrow\rho_d(g)(g_1,\dots,g_d)$ we have $g_d=g$.
\end{definition}

\begin{lemma}\label{lem:persistent}
For any self-similar group $G\le \Aut(\tree_{d-1})$, there is a faithful, persistent, self-similar action of $G$ on $\tree_d$. If the action on $\tree_{d-1}$ is rational then so is the action on $\tree_d$. If the action on $\tree_{d-1}$ is weakly diagonal then so is the action on $\tree_d$.
\end{lemma}

\begin{proof}
For $g\in G$ write the wreath recursion of $g$ as $g\leftrightarrow\rho_{d-1}(g)(g_1,\dots,g_{d-1})$. Now define an action of $G$ on $\tree_d$ by first extending the action $\rho_{d-1}$ of $G$ on $\{1,\dots,d-1\}$ to an action $\rho_d$ of $G$ on $\{1,\dots,d\}$ by fixing $d$, and then recursively defining an action on all of $\tree_d$ via the wreath recursion $g\leftrightarrow\rho_d(g)(g_1,\dots,g_{d-1},g)$. This new action is persistent by construction, so we just need to prove that it is faithful. Note that there is a natural copy of $\tree_{d-1}$ inside $\tree_d$, coming from the natural inclusion of $\{1,\dots,d-1\}$ into $\{1,\dots,d\}$, and the action of $G$ on $\tree_d$ stabilizes $\tree_{d-1}$. The restriction of the new action of $G$ on this copy of $\tree_{d-1}$ is equal to the original action, and so we conclude that since the action of $G$ on $\tree_{d-1}$ is faithful, so is the action of $G$ on $\tree_d$. The set of states of a given element under the action on $\tree_{d-1}$ is equal to the set of states of that element under the action on $\tree_d$, so if the original action is rational, so is the new action. If the original action is weakly diagonal, then it is trivial to see that the new action is too.
\end{proof}

\begin{corollary}\label{cor:simple_and_qr}
Let $G\le\Aut(\tree_d)$ be a finitely generated, persistent, weakly diagonal, rational, self-similar group. Then $V_d(G)$ is virtually simple by virtue of $[V_d(G),V_d(G)]$ being simple and finite index, and there exists a quasi-retraction $V_d(G)\to G$. If $G$ is finitely presented, then so is $V_d(G)$ and we have $\delta_G\preceq \delta_{V_d(G)}$.
\end{corollary}

\begin{proof}
Lemma~\ref{lem:v_simple} says that $[V_d(G),V_d(G)]$ is simple and finite index. The existence of a quasi-retraction follows from \cite[Proposition~5.5]{skipper19}. If $G$ is finitely presented, then so is $V_d(G)$, for example by \cite[Theorem~4.15]{skipper19}. That $\delta_G \preceq \delta_{V_d(G)}$ now follows from Citation~\ref{cit:qr_dehn}.
\end{proof}

%-------------------------------------------------
\section{Our examples}\label{sec:examples}

Now we can build our examples of finitely presented simple groups with at least exponential Dehn function. They will arise as commutator subgroups of R\"over--Nekrashevych groups where the self-similar input group is a Baumslag--Solitar group. Recall that the \emph{Baumslag--Solitar group} $BS(m,n)$ is the group
\[
BS(m,n)\coloneqq \langle a,b\mid ab^m a^{-1} = b^n\rangle \text{.}
\]
For most values of $(m,n)$, the group $BS(m,n)$ does not stand a chance of being self-similar, since it is not even residually finite. (It is easy to see that $\Aut(\tree_d)$ is residually finite, and hence so are all self-similar groups.) Thus, we focus on the residually finite case of $m=1$, where moreover $BS(1,n)$ is known to be self-similar \cite{bartholdi06}. It is well known that $BS(1,n)$ has exponential Dehn function for all $n\ge 2$ \cite{gersten92}.

Let us explicitly realize $BS(1,n)$ ($n\ge 2$) as a self-similar group, using wreath recursions inspired by the automata in \cite{bartholdi06}. We want to define an action of $BS(1,n)$ on $\tree_{n+1}$. First let $\rho_{n+1}\colon BS(1,n)\to S_{n+1}$ send $a$ to $\alpha\coloneqq (2~n+1)(3~n)(4~n-1)\cdots$ and $b$ to $\beta\coloneqq (1~2~\cdots~n+1)$. To be more precise, if $n$ is even then $\alpha$ ends with $(\frac{n}{2}+1 ~ \frac{n}{2}+2)$ and if $n$ is odd then it ends with $(\frac{n-1}{2}+1 ~ \frac{n-1}{2}+3)$. It is easy to check that $\alpha\beta\alpha^{-1}=\beta^n$, so $\rho_{n+1}$ is well defined. Now we extend this to an action on all of $\tree_{n+1}$ via the wreath recursions
\[
a\leftrightarrow \alpha(a,a,ba,b^2a,\dots,b^{n-1}a)\text{ and }
b\leftrightarrow \beta(1,\dots,1,b) \text{.}
\]

\begin{proposition}\label{prop:check_everything}
The self-similar action of $BS(1,n)$ on $\tree_{n+1}$ defined by the above wreath recursions is well defined, faithful, rational, and weakly diagonal.
\end{proposition}

\begin{proof}
To check that the action is well defined, we need to confirm that the words $ab$ and $b^n a$ have the same action on $\tree_d$. Concatenating the wreath recursions, we compute that $ab$ corresponds to
\begin{align*}
\alpha(a,a,ba,b^2a,\dots,b^{n-1}a)\beta(1,\dots,1,b) &= \alpha\beta(a,ba,b^2a,\dots,b^{n-1}a,a)(1,\dots,1,b) \\
&= \alpha\beta(a,ba,b^2a,\dots,b^{n-1}a,ab)
\end{align*}
and $b^n a$ corresponds to
\begin{align*}
(\beta(1,\dots,1,b))^n \alpha(a,a,ba,b^2a,\dots,b^{n-1}a) &=
\beta^n(1,b,\dots,b) \alpha(a,a,ba,b^2a,\dots,b^{n-1}a) \\
&= \beta^n\alpha(1,b,\dots,b)(a,a,ba,b^2a,\dots,b^{n-1}a) \\
&= \alpha\beta(a,ba,b^2a,\dots,b^{n-1}a,b^na)\text{.}
\end{align*}
Thus we see that the wreath recursion of $b^{-1}a^{-1}b^n a$ is $b^{-1}a^{-1}b^n a \leftrightarrow (1,\dots,1,b^{-1}a^{-1}b^n a)$, which means $b^{-1}a^{-1}b^n a$ acts trivially on $\tree_{n+1}$ as desired.

To check that the action is faithful, we will use the fact that every non-trivial normal subgroup of $BS(1,n)$ contains a non-trivial power of $b$. Indeed, any element can be written in the form $a^{-k}b^q a^\ell$ for $k,\ell\ge 0$ and $q\in\Z$, so if a normal subgroup contains this element, then conjugating by $a$ it must also contain $a^{-k}b^{nq} a^\ell$, hence $a^{-k}b^{(n-1)q} a^k$, hence $b^{(n-1)q}$. Now if our original element is non-trivial, then either $q\ne 0$, so $b^{(n-1)q}$ is non-trivial and we are done, or else $q=0$ and our non-trivial element was $a^{\ell-k}$. Thus our normal subgroup contains $a^{\ell-k} b a^{-(\ell-k)} b^{-1} = b^{n^{\ell-k} - 1}$, which is a non-trivial power of $b$. Now to see that the action is faithful, we just need to show that no non-trivial power of $b$ acts trivially on $\tree_{n+1}$. Indeed, for any $k>0$ the element $b^k$ has $b$ as a state, and so acts non-trivially on $\tree_{n+1}$.

Now we prove that the action is rational. We just need to check that the generators $a$ and $b$ have finitely many states. Clearly $1$ and $b$ are the only states of $b$. We claim that the level-1 states of $a$ comprise all the states of $a$, i.e., we claim that the set $\{a,ba,b^2 a,\dots,b^{n-1} a\}$ is state-closed. For any $0\le k\le n-1$, we have the wreath recursion $b^k \leftrightarrow \beta^k (1,\dots,1,b,\dots,b)$, where the total number of entries equal to $b$ is $k$. Thus, $b^k a$ corresponds to
\begin{align*}
&\beta^k (1,\dots,1,b,\dots,b) \alpha (a,a,ba,b^2 a,\dots,b^{n-1}a) \\ =  &\beta^k \alpha (1,b,\dots,b,1,\dots,1)(a,a,ba,b^2 a,\dots,b^{n-1}a) \text{,}
\end{align*}
where the total number of entries in the first tuple equal to $b$ is $k$. Since $k\le n-1$, this equals $\beta^k \alpha (a,ba,b^2 a,b^3 a,\dots,b^k a,b^k a,\dots,b^{n-1}a)$. This shows that every level-1 state of an element of the set $\{a,ba,b^2 a,\dots,b^{n-1} a\}$ lies again in this set, and so the set is state-closed.

Finally, we need to show that this action is weakly diagonal. Note that $b^{n-1}=aba^{-1}b^{-1}$ is trivial in the abelianization of $BS(1,n)$. Thus, every power of $b$ has finite order in the abelianization. By looking at the wreath recursions of $a$ and $b$, it is clear that for any level-1 state $a_i$ of $a$ the element $a_ia^{-1}$ is a power of $b$, and for any level-1 state $b_i$ of $b$ the element $b_ib^{-1}$ is a power of $b$ (namely $1$ or $b^{-1}$). Thus, indeed the action is weakly diagonal.
\end{proof}

Note that the action is not coarsely diagonal, e.g., $a_3=ba$, so $a_3 a^{-1}=b$ has infinite order. Thus, we could not directly use the results of \cite{skipper19}, and it was necessary to introduce the new notion of weakly diagonal.

Now all the pieces are in place to prove our main result.

\begin{theorem}\label{thrm:exist_simple}
There exist finitely presented simple groups with at least exponential Dehn function.
\end{theorem}

\begin{proof}
View $BS(1,n)$ as a self-similar group in $\Aut(\tree_{n+1})$ as above. By Proposition~\ref{prop:check_everything}, $BS(1,n)$ is rational and weakly diagonal. Now use Lemma~\ref{lem:persistent} to view $BS(1,n)$ as a persistent self-similar group in $\Aut(\tree_{n+2})$, which is still rational and weakly diagonal. By Corollary~\ref{cor:simple_and_qr}, $V_{n+2}(BS(1,n))$ is finitely presented with $\delta_{BS(1,n)}\preceq \delta_{V_{n+2}(BS(1,n))}$, so the Dehn function of $V_{n+2}(BS(1,n))$ is at least exponential. Corollary~\ref{cor:simple_and_qr} also says that the commutator subgroup $[V_{n+2}(BS(1,n)),V_{n+2}(BS(1,n))]$ has finite index in $V_{n+2}(BS(1,n))$, and is simple. Thus, we conclude that $[V_{n+2}(BS(1,n)),V_{n+2}(BS(1,n))]$ is a finitely presented simple group whose Dehn function is at least exponential.
\end{proof}

\bibliographystyle{alpha}

\end{document}